\documentclass[12pt,a4paper]{article}
\usepackage{amsmath, amsthm, amssymb}
\usepackage{fullpage, verbatim, graphicx, bbm, enumitem}

\newcommand{\todo}[1]{\vspace{5 mm}\par \noindent
\marginpar{\textsc{ToDo}} \framebox{\begin{minipage}[c]{0.95
\textwidth} \flushleft \tt #1 \end{minipage}}\vspace{5 mm}\par}

\newcommand{\note}[1]{\vspace{5 mm}\par \noindent
\marginpar{\textsc{Note}} \framebox{\begin{minipage}[c]{0.95
\textwidth} \flushleft \tt #1 \end{minipage}}\vspace{5 mm}\par}

\newcommand{\waven}{\lambda}

\newcommand{\R}{\mathbb{R}}
\renewcommand{\L}{\mathbb{L}}
\renewcommand{\S}{\mathbb{S}}
\newcommand{\T}{S^1}
\newcommand{\I}{\mathrm{i}}

\newcommand*{\V}{\mathcal{V}}
\newcommand*{\W}{\mathcal{W}}

\newcommand*{\modi}{{\operatorname{mod}}}
\newcommand*{\lin}{{\operatorname{lin}}}
\newcommand*{\nl}{{\operatorname{nl}}}

\newcommand*{\Span}{\operatorname{span}}

\renewcommand*{\d}{{\mathrm d}}

\newtheorem{theorem}{Theorem}
\newtheorem{corollary}[theorem]{Corollary}
\newtheorem{lemma}[theorem]{Lemma}
\theoremstyle{remark}

\newcommand*{\tnorm}[3][]{\lVert #2 \rVert_{#3}^{#1}}
\newcommand*{\norm}[3][{\vphantom 1}]{\lVert #2 \rVert_{#3}^{#1}}

\title{A new construction of modified equations for variational
integrators}

\author{Marcel Oliver\thanks{Mathematical Institute for
Machine Learning and Data Science, KU Eichst\"att--Ingolstadt, 85049
Ingolstadt, Germany and Constructor University, 28759 Bremen, Germany
(marcel.oliver@ku.de)}
\and Sergiy Vasylkevych\thanks{Institute of
Meteorology, Universität Hamburg, 20144 Hamburg, Germany
(sergiy.vasylkevych@uni-hamburg.de)}}

\date{\today}

\begin{document}
\maketitle
\let\thefootnote\relax\footnotetext{\emph{MSC Classification}: 65P10,
70H03, 35L05}

\begin{center}
\textit{In memory of Claudia Wulff}
\end{center}

\begin{abstract}
The construction of modified equations is an important step in the
backward error analysis of symplectic integrator for Hamiltonian
systems.  In the context of partial differential equations, the
standard construction leads to modified equations with increasingly
high frequencies which increase the regularity requirements on the
analysis.  In this paper, we consider the next order modified
equations for the implicit midpoint rule applied to the semilinear
wave equation to give a proof-of-concept of a new construction which
works directly with the variational principle.  We show that a
carefully chosen change of coordinates yields a modified system which
inherits its analytical properties from the original wave equation.
Our method systematically exploits additional degrees of freedom by
modifying the symplectic structure and the Hamiltonian together.
\end{abstract}

\section{Introduction}

Over the last decades, backward error analysis has emerged as a useful
tool for proving conservation properties of numerical time integrators
for differential equations.  In a nutshell, one constructs a modified
differential equation which is approximated by the numerical method to
some higher order than the original equation and which possesses
analogous conservation laws. Thus, the numerical scheme is nearly
conservative on time scales over which it approximates the solution of
the modified equation.  Some of the strongest and most general results
for ordinary differential equations were obtained by Benettin and
Giorgilli \cite{BenettiinG:1994:HamiltonianIN} and Hairer and Lubich
\cite{HairerL:1997:LifeSB}, using ideas which go back to Neishtadt
\cite{Neishtadt:1984:SeparationMS}, who optimally truncate an
asymptotic series for the modified vector field to prove that a class
of symplectic schemes when applied to Hamiltonian ordinary
differential equations preserve the energy exponentially well over
exponentially long times in the step size.  Similar ideas appear in
\cite{HairerLW:2006:GeometricNI, LeimkuhlerR:2004:SimulatingHD,
Reich:1999:BackwardEA}.

While this construction formally extends to Hamiltonian partial
differential equations, in this case the asymptotic series for the
modified equation will generally contain arbitrary powers of unbounded
operators, thereby breaking the natural ordering of the terms in the
asymptotic series.  In particular, the standard construction in the
context of hyperbolic equations such as the semilinear wave equation,
fails in the practically relevant regime when the scaling of time vs.\
space step is close to the CFL limit.  This problem has been partially
addressed in a number of ways.  Cano \cite{Cano:2006:ConservedQH}
proves an exponential backward error analysis result conditional on a
number of conjectures.  Moore and Reich \cite{MooreR:2003:BackwardEA}
and Islas and Schober \cite{IslasS:2005:BackwardEA} provide a formal
backward error analysis in a multisymplectic setting.  Under strong
regularity assumptions on the true solution,
\cite{OliverW:2012:AstableRK, OliverW:2014:StabilityGT,
WulffO:2016:ExponentiallyAH} show that the occurrence of unbounded
operators in the modified vector field only leads to loss of order in
the exponents of the backward error estimates.  Sometimes,
non-standard modified equations can be helpful, such as in the result
of \cite{OliverWW:2004:ApproximateMC} on the approximate numerical
preservation of the momentum invariant.  Cohen, Hairer, and Lubich
\cite{CohenHL:2008:ConservationEM} obtained results on polynomially
long times for full time-spectral discretization of weakly nonlinear
wave equations using the method of modulated Fourier expansions.
Finally, there is a large body of work on the backward error analysis
of splitting methods applied to partial differential equations
\cite{DujardinF:2007:NormalFL, DebusscheF:2009:ModifiedES,
FaouGP:2010:BirkhoffFD, FaouGP:2010:BirkhoffAS,
GaucklerL:2010:SplittingIN}.  However, the question whether strong
results for generic solutions, i.e.\ solutions which are neither small
nor analytic, can be obtained remains open.

Modified equations are clearly not unique.  Various expressions
appearing beyond the leading order of the asymptotic series can be
consistently replaced by using the modified equation itself.  In
principle, one can use such substitutions to remove the occurrence of
high frequencies at high orders of the modified system.
Unfortunately, this process will generally break the Hamiltonian
structure.

This paper is motivated by the observation that a large class of
symplectic integrators can be derived via a discrete variational
principle.  Elementary examples appear in Wendlandt and Marsden
\cite{WendlandtM:1997:MechanicalID} and Marsden and West
\cite{MarsdenW:2001:DiscreteMV} who, in particular, show that the
implicit midpoint rule arises via a simple finite difference
approximation of the action integral.  More generally, a large number
of symplectic schemes arises as variational integrators
\cite{LeokS:2012:GeneralTC, MarsdenW:2001:DiscreteMV}; in particular,
Leok and Zhang \cite{LeokZ:2011:DiscreteHV} show that it is also
possible to obtain variational integrators on the Hamiltonian side,
which extends the concept to systems with degenerate Hamiltonians.
Vermeeren \cite{Vermeeren:2017:ModifiedEV,Vermeeren:2019:ModifiedEV}
modified equations on the Lagrangian side, albeit only for finite
dimensional systems and in the sense of so-called meshed Lagrangians.
Recently, McLachlan and Offen \cite{McLachlanO:2022:BackwardEA}
consider variational backward error analysis for symmetry solutions of
wave equations.  This reduction again specializes the problem to
finite-dimensional modified equations.  A variational treatment of the
general case is, so far, open.

In this paper, we demonstrate that it is possible to construct
modified equations via a classical variational principle.  A naive
variational construction has the drawback that the order of time
derivatives in the modified equations increases with the order to
which the modified equations are constructed, i.e., the phase space
gets increasingly larger.  Since the higher time derivatives appear at
higher orders of the expansion (see e.g. \cite{MooreR:2003:BackwardEA,
McLachlanO:2022:BackwardEA}), such constructions lead to singular
perturbation problems with multiple fast time scales.  Still, the
original slow dynamics lives on a submanifold in this larger phase
space.  Our main point is that we can approximately restrict to this
submanifold by the use of a near-identity change of variables which
moves all fast degrees of freedom beyond the truncation order of the
asymptotic expansion.  We call this approach the \emph{method of
degenerate variational asymptotics}; it is motivated by earlier work
on model reduction for rapidly rotating fluid flow
\cite{Oliver:2006:VariationalAR, OliverV:2011:HamiltonianFM}.

The current work is the first proof-of-concept for this approach.  In
the first part of the paper, we consider a simple, yet nontrivial
special case: the next-order modified equations for the implicit
midpoint scheme applied to the semilinear wave equation.  We find that
the new modified equations do not admit frequencies beyond the scale
already present in the original partial differential equation.  In
particular, unlike the modified equations which arise from the
conventional construction, they have a dispersion relation for linear
waves which has a finite limit as the wave number $k$ tends to
infinity.  This behavior coincides qualitatively with that of the
implicit midpoint rule itself, which also possesses a finite highest
numerical frequency in a time-semidiscrete analysis
\cite{BridgesR:2006:NumericalMH}.  Moreover, we can show that the full
nonlinear modified equations are well-posed---locally in time but for
a time interval which is independent of the time step
parameter---precisely in the energy space which arises naturally from
the modified Hamiltonian.  In the second part of the paper, we show
how to extend the approach to arbitrary high order.  This part of the
paper is formal, but we make sure that all spatial operators appearing
in the final modified equations are bounded, and all such operators,
except those that must limit to the second space derivative as
$h\to0$, are bounded uniformly as a function of $h$.  We start with
looking at the linear wave equation only, where an all-order modified
equation can be found via a generating function approach.  From there,
we move to the nonlinear case where, up to an arbitrary but fixed
order, we use a bilevel iterative concatenation of transformations to
remove all higher-order time derivatives from the Lagrangian.

The paper is structured as follows.  After the introduction of the
semilinear wave equation in Section~\ref{s.wave},
Section~\ref{s.var-int} gives a brief derivation of the implicit
midpoint scheme as a variational integrator.  In
Section~\ref{s.bea-ham}, we recall the standard Hamiltonian
construction of the modified vector field and show that the result is
only useful under restrictive time-step assumptions.
Section~\ref{s.bea-lagrange} explains the naive variational
construction; Section~\ref{s.deg} introduces the method of degenerate
variational asymptotics, which constrains the phase space of the
modified equations to the slow degrees of freedom.  After addressing
the question of consistent initialization of the modified system in
Section~\ref{s.init}, in Section~\ref{s.num} we give a numerical
evidence that the new modified equations indeed perform as claimed.
In Section~\ref{s.wellposed}, we present an analytic framework in
which the nonlinear elliptic operator which arises in the formulation
of the new modified equations is invertible, and we recast them in the
form of a semilinear evolution equation, thereby obtaining a proof of
local well-posedness in the natural energy space.
Section~\ref{s.higher-linear} discusses the all-order modified
equations for the linear wave equations, and Section~\ref{s.higher-nl}
the subsequent extension to the nonlinear case.  Finally,
Section~\ref{s.conclusions} concludes with a brief discussion of the
results.

\section{The semilinear wave equation}
\label{s.wave}

We consider the semilinear wave equation on the circle $\T$,
\begin{equation}
  \label{e.wave}
  \ddot u = \partial_{xx} u + f(u) \,,
\end{equation}
where $u=u(x,t)$ and we write $\dot u \equiv \partial_t u$.  It arises
as the Euler--Lagrange equation with Lagrangian $L \colon Q \times Q
\rightarrow \R$ given by
\begin{equation}
  L(u, \dot u ) = \int_{\T}
    \tfrac12 \dot u^2 - \tfrac12 \, (\partial_xu)^2 + V(u) \, \d x
\end{equation}
with $f = V'$ and $Q$ being a space of sufficiently smooth functions
on $\T$; we also assume that $V$ is smooth.

Writing $p = \dot u$, the semilinear wave equation is Hamiltonian with
conserved energy
\begin{equation}
  H(p,u)
  = \int_{\T} \tfrac1{2} \, p^2
    + \tfrac1{2} \, (\partial_xu)^2 - V(u) \, \d x \,.
\end{equation}
In addition, the Lagrangian is invariant under space translations.
Hence, Noether's theorem implies conservation of momentum
\begin{equation}
  \label{e.momentum}
  J(u,p)= \int_{\T} p \, \partial_xu \, \d x \,.
\end{equation}

It is often convenient to write the semilinear wave equation as a
first order system: setting
\begin{equation}
  U =
  \begin{pmatrix}
    u \\ p
  \end{pmatrix} \,, \qquad
  A =
  \begin{pmatrix}
    0 & 1 \\ \partial_{xx} & 0
  \end{pmatrix} \,, \qquad \text{and} \qquad
  B(U) =
  \begin{pmatrix}
    0 \\ f(u)
  \end{pmatrix} \,,
\end{equation}
equation \eqref{e.wave} takes the form
\begin{equation}
  \dot U = A U + B(U) \,.
  \label{e.wave2}
\end{equation}

\section{Variational integrators}
\label{s.var-int}

Consider uniform grid on time interval $[0,T]$ with mesh size $h=T/n$.
Following \cite{Veselov:1988:IntegrableSD,
WendlandtM:1997:MechanicalID}, we consider the discrete variational
principle associated with the temporal semidiscretization, namely,
find $u_0, \dots, u_{n}$ which are a stationary point of the discrete
action
\begin{equation}
  \label{e.d.action}
  \S  = \sum_{k=0}^{n-1} \L (u_k, u_{k+1}; h) \,,
\end{equation}
subject to variations which leave the temporal endpoints $u_0$ and
$u_n$ fixed.  The discrete variational principle $\delta \S(u_1,
\dots, u_{n-1}) = 0$ yields the discrete Euler--Lagrange equation
\begin{equation}
  \label{e.d.EL}
  D_1 \L (u_k,u_{k+1};h) + D_2 \L (u_{k-1},u_k;h) = 0
\end{equation}
for $k=1, \dots, n-1$.

A symplectic scheme of second order for the semilinear wave equation
is obtained by taking the discrete Lagrangian
\begin{equation}
  \label{e.d.lagr}
  \L(u_{k}, u_{k+1}; h)
  = h \, L \Bigl(
             \frac{u_k+u_{k+1}}{2}, \frac{u_{k+1}- u_k}{h}
           \Bigr) \,.
\end{equation}
Introducing the discrete Legendre transform
\cite[p.~194]{HairerLW:2006:GeometricNI},
\begin{align}
  p_k
  & \equiv - D_1 \L (u_k,u_{k+1}; h)
    \notag \\
  & = \frac{\delta L}{\delta \dot u}
      \biggl( \frac{{u_k}+u_{k+1}}{2}, \frac{u_{k+1}-u_{k}}{h} \biggr)
      - \frac{h}{2} \, \frac{\delta L}{\delta u}
      \biggl( \frac{{u_k}+u_{k+1}}{2}, \frac{u_{k+1}-u_{k}}{h} \biggr) \,,
  \label{e.d.legendre}
\end{align}
and using the discrete Euler--Lagrange equation \eqref{e.d.EL}, we
obtain
\begin{align}
  p_{k+1}
  & = D_2 \L (u_{k},u_{k+1}; h)
    \notag \\
  & = \frac{\delta L}{\delta \dot u}
      \biggl( \frac{{u_k}+u_{k+1}}{2}, \frac{u_{k+1}-u_{k}}{h} \biggr)
      + \frac{h}{2} \, \frac{\delta L}{\delta u}
      \biggl( \frac{{u_k}+u_{k+1}}{2}, \frac{u_{k+1}-u_{k}}{h} \biggr) \,.
  \label{e.p.advance}
\end{align}
The variations on the right of \eqref{e.d.legendre} and
\eqref{e.p.advance} read
\begin{gather}
  \frac{\delta L}{\delta \dot u}(u,\dot u) = \dot u
  \qquad \text{and} \qquad
  \frac{\delta L}{\delta u}(u,\dot u) = \partial_{xx}u + f(u) \,,
  \label{e.L.partial}
\end{gather}
where we identify $Q$ with a subspace of $Q^*$ via the $L^2$ inner
product.  We now introduce an intermediate integration stage
via
\begin{gather}
  u_{k+1/2} = \frac{u_k+u_{k+1}}2
  \qquad \text{and} \qquad
  p_{k+1/2} = \frac{p_k+p_{k+1}}2 \,.
  \label{e.intermediate}
\end{gather}
Then, taking the sum and difference of \eqref{e.d.legendre} and
\eqref{e.p.advance}, respectively, we obtain
\begin{subequations}
  \label{e.im1}
\begin{gather}
  u_{k+1} = u_k + h \, p_{k+1/2} \,, \\
  p_{k+1} = p_k + h \,
            \bigl( \partial_{xx} u_{k+1/2} + f(u_{k+1/2}) \bigr) \,.
\end{gather}
\end{subequations}
Written in this form, the scheme is clearly recognized as the implicit
midpoint rule, which is, in fact, a second order Gauss--Legendre
Runge--Kutta method.  To obtain a practical numerical scheme, it is
better to replace the definition of the intermediate integration stage
by the equivalent expressions \cite{BrennerCT:1982:SingleSM}
\begin{subequations}
  \label{e.intermediate2}
\begin{gather}
  u_{k+1/2} = u_k + \tfrac{h}{2} \, p_{k+1/2} \,, \\
  p_{k+1/2} = p_k + \tfrac{h}{2} \,
              \bigl( \partial_{xx} u_{k+1/2} + f(u_{k+1/2}) \bigr) \,.
\end{gather}
\end{subequations}
In terms of the vector notation introduced at the end of
Section~\ref{s.wave}, \eqref{e.intermediate2} reads
\begin{gather}
  \label{e.intermediate3}
  U_{k+1/2} = U_k + \tfrac{h}2 \,
              \bigl( AU_{k+1/2} + B(U_{k+1/2}) \bigr) \,,
\end{gather}
or
\begin{gather}
  \label{e.intermediate4}
  U_{k+1/2} = (1 - \tfrac{h}2A)^{-1} \bigl(U_k + \tfrac{h}2 \,
              B(U_{k+1/2}) \bigr) \,.
\end{gather}
For sufficiently small $h$ and a suitable choice of function space,
the operator on the right side of \eqref{e.intermediate4} is a
contraction, so that the intermediate stage vector $U_{k+1/2}$ can be
found iteratively.

Similarly, noting that $1+hA \, (1 - \tfrac{h}2A)^{-1} = (1 +
\tfrac{h}2A)(1 - \tfrac{h}2A)^{-1}$, we can write \eqref{e.im1} in the
form
\begin{gather}
  \label{e.im2}
  U_{k+1} = S(hA) U_k + h \, (1 - \tfrac{h}2A)^{-1} \,
              B(U_{k+1/2}) \,,
\end{gather}
where
\begin{gather}
  S(z) = {\left(1+z/2\right)}{\left(1-z/2 \right)^{-1}}
\end{gather}
is known as the \emph{stability function} of the method.  Again,
equation \eqref{e.im2} considered, for instance, in a Sobolev space
$H^s(\T)\times H^{s-1}(\T)$ with $s\geq 1$ has only bounded operators on its right hand side, so the
time-$h$ map of the scheme does not lose derivatives.  For this
reason, the system \eqref{e.intermediate4} and \eqref{e.im2} is the
preferred form for numerical implementation of the implicit midpoint
scheme.

\section{Backward error analysis on the Hamiltonian side}
\label{s.bea-ham}

In this section, we give an elementary derivation of the standard
Hamiltonian modified equation which corresponds to the implicit
midpoint rule up to terms of $O(h^4)$.  The procedure is that of
Hairer, Lubich, and Wanner \cite[Chapter
IX]{HairerLW:2006:GeometricNI}; our presentation reduces the procedure
to the special case at hand.

We use the following notation.  Let $U \colon [0,T] \to Q \times Q^*$
be a curve satisfying an autonomous equation of the form
\begin{equation}
  \dot{U}=F(U)
\end{equation}
with $U(kh) = U_k$ for $kh \in [0,T]$.  Let us fix $k$ and Taylor
expand about $t_0=kh$.  On the one hand,
\begin{equation}
  \label{e.taylor0}
  U_{k+1} = U(t_0+h)
          = U_k + \dot U_k \, h + \tfrac12 \,\ddot U_k \, h^2
            + \tfrac16 \, U_k^{(3)} \, h^3 + O(h^4) \,,
\end{equation}
where $\dot U_k \equiv U(kh)$, $\ddot U_k \equiv \ddot U(kh)$, and
$U_k^{(3)} \equiv U^{(3)}(kh)$.  On the other hand, $U_{k+1}$ is
determined by the implicit midpoint rule
\begin{align}
  U_{k+1}
  & = U_k + h \, F(U_{k+1/2})
      \notag \\
  & = U_k + F(U_k) \, h + \tfrac12 \, F'(U_k) \, \dot U_k \, h^2
      + \tfrac18 \, F''(U_k)(\dot U_k, \dot U_k) \, h^3
      + \tfrac14 \, F'(U_k) \, \ddot U_k \, h^3 + O(h^4) \,.
  \label{e.imid-expanded}
\end{align}

Since the implicit midpoint rule is a symmetric method, only terms at
even powers of $h$ appear in the modified equation
\cite{HairerLW:2006:GeometricNI}.  Therefore, to determine the
modified equation including terms of $O(h^2)$, we seek a vector field
$F_2$ such that
\begin{equation}
  \dot U_k = F(U_k) + F_2(U_k) \, h^2 + O(h^3) \,.
  \label{e.mod-abstract}
\end{equation}
We note that \eqref{e.mod-abstract} implies the approximate identities
\begin{subequations}
  \label{e.mod-abstract-diff}
\begin{gather}
  \ddot U_k = F'(U_k) \, \dot U_k + O(h^2) \,, \\
  U_k^{(3)}
  = F''(U_k) (\dot U_k, \dot U_k) + F'(U_k) \, \ddot U_k + O(h^2) \,.
\end{gather}
\end{subequations}
Then, equating \eqref{e.taylor0} and \eqref{e.imid-expanded}, using
\eqref{e.mod-abstract} and \eqref{e.mod-abstract-diff} to eliminate
all derivatives, we obtain
\begin{gather}
  F_2(U_k) = \tfrac1{12} \, F'(U_k) \, F'(U_k) \, F(U_k)
           - \tfrac1{24} \, F''(U_k) (F(U_k),F(U_k)) \,.
\end{gather}
In the particular case of the semilinear wave equation, where
\begin{equation}
  F(U) = AU + B(U) \,,
\end{equation}
we obtain by direct calculation that
\begin{equation}
  F_2(U) = \frac1{12} \,
  \begin{pmatrix}
    \partial_{xx} p + f'(u) p \\
    (\partial_{xx} + f'(u))(\partial_{xx} u + f(u))
  \end{pmatrix}
  - \frac1{24} \,
  \begin{pmatrix}
  0 \\ f''(u) \, p^2
  \end{pmatrix} \,.
\end{equation}
Hence, the modified equation up to terms of order
$O(h^4)$ reads
\begin{subequations}
  \label{e.PDE.mod1}
\begin{gather}
  \dot u = \bigl[
             1 + \tfrac1{12} \, h^2 \, (\partial_{xx}+f'(u))
           \bigr] \, p \,,
  \label{e.PDE.mod1a} \\
  \dot p = \bigl[
             1 + \tfrac1{12} \, h^2 \, (\partial_{xx} + f'(u))
           \bigr] \, (\partial_{xx} u + f(u))
           - \tfrac1{24} \, h^2 \, f''(u) \, p^2 \,.
\end{gather}
\end{subequations}
It is straightforward to verify that modified equations
\eqref{e.PDE.mod1} define a Hamiltonian system with Hamiltonian
\begin{equation}
  \label{e.H.mod}
  H_\modi (p,u) = H(p,u) + \tfrac1{24} \, h^2 \,
  \int_{\T} f'(u) \, p^2 - (\partial_x p)^2 -
             (\partial_{xx} u +f(u))^2 \, \d x \,.
\end{equation}
It is generally true that symplectic Runge--Kutta schemes applied to
Hamiltonian systems yield a Hamiltonian modified equation at any order
\cite[Section~IX.3]{HairerLW:2006:GeometricNI}.

The difficulty with using the modified system \eqref{e.PDE.mod1} for
the purpose of backward error analysis can be seen as follows.  We
consider, for simplicity, the linear case when $f\equiv 0$.  Writing
$\waven$ to denote the spatial wave number, equation \eqref{e.wave2} in the
space-frequency domain reads
\begin{equation}
  \partial_t \hat U_{\waven}
  = (1 - \tfrac1{12} \, h^2 \, \waven^2) \,
  \begin{pmatrix}
    0 & 1 \\ - {\waven}^2 & 0
  \end{pmatrix} \, \hat U_{\waven} \,.
  \label{e.lin-ham}
\end{equation}
We observe that the additional factor $h^2 \, \waven^2$ will introduce fast
frequencies into the modified dynamics unless we restrict the
admissible wave numbers and time steps.  More generally, it can be
shown that the modified vector field will be an asymptotic series in
$h \waven$, so that the series will be properly ordered only if we restrict
to a finite dimensional subspace of wave numbers---a discretization in
space---and if $h = o(\waven_{\max}^{-1})$.  This excludes the practically
relevant regime of time steps as large as permitted by the CFL
condition where $h \sim \waven_{\max}^{-1}$; see, for example, the
discussion in \cite{CohenHL:2008:ConservationEM}.

\section{Backward error analysis on the Lagrangian side}
\label{s.bea-lagrange}

We now turn to the question of how we can derive modified equations
using the variational principle which underlies the semilinear wave
equation.

The notation here is analogous to the notation used in the previous
section: let $u \colon [0,T] \mapsto Q$ be a curve in $Q$ such that $u(kh) =
u_k$ for $kh \in [0,T]$.  Again, we fix $k$ and Taylor expand  $u$ and
$\dot u$  about $t_0=kh$.  Then, for $\lvert t \rvert \leq h$,
\begin{subequations}
  \label{e.taylor}
\begin{gather}
  \label{e.taylor.u}
  u(kh+t) = u_k + \dot u_k \, t + \tfrac1{2} \, \ddot u_k \, t^2
    + \tfrac1{6} \, u_k^{(3)} \, t^3 + O(h^4) \,, \\
  \dot u (kh+t) = \dot u_k + \ddot u_k \, t
    + \tfrac1{2} \, u_k^{(3)} \, t^2 + O(h^3) \,,
\end{gather}
\end{subequations}
where $\dot u_k \equiv \dot u(kh)$, $\ddot u_k \equiv \ddot u(kh)$, and
$u^{(3)}_k \equiv u^{(3)}(kh)$.  In particular, setting $t=h$, we have
\begin{equation}
  \label{e.uk1}
  u_{k+1}= u_k + \dot u_k \, h + \tfrac1{2} \, \ddot u_k \, h^2
    + \tfrac1{6} \, u_k^{(3)} \, h^3 + O(h^4) \,.
\end{equation}
Substituting \eqref{e.uk1} into the discrete Lagrangian, we obtain
\begin{align}
  \L(u_k, u_{k+1};h)
  & = h \, L \bigl(u_k + \tfrac1{2} \, h \, \dot u_k
      + \tfrac1{4} \, {h^2} \, \ddot u_k + O(h^3),
      \dot u_k + \tfrac1{2} \, h \, \ddot u_k
      + \tfrac1{6} \, {h^2} \, u^{(3)}_k + O(h^3) \bigr)
    \notag \\
  & = h \int_{\T} \tfrac1{2} \,
      \bigl(
        \dot u_k + \tfrac1{2} \, h \, \ddot u_k
        + \tfrac1{6} \, {h^2} \, u^{(3)}_k
      \bigr)^2
       - \tfrac1{2} \,
      \bigl(
        \partial_xu_k + \tfrac1{2} \, h \, \partial_x \dot u_k
        + \tfrac1{4} \, {h^2} \, \partial_x \ddot u_k
      \bigr)^2 \, \d x
    \notag \\
  & \quad + h \int_{\T} V(u_k) + f(u_k) \,
      \bigl(
        \tfrac1{2} \, h \, \dot u_k + \tfrac1{4} \, {h^2} \, \ddot u_k
      \bigr)
      + \tfrac1{8} \, h^2 \, f' (u_k) \, \dot u_k^2 \, \d x + O(h^4)
    \notag \\
  & = h \, L(u_k, \dot u_k)
      + \tfrac1{2} \, h^2 \int_{\T} \ddot u_k \, \dot u_k -
        \partial_x u_k \, \partial_x \dot u_k + f(u_k) \, \dot u_k \, \d x
    \notag \\
  & \quad + h^3 \int_{\T}
      \biggl[
        \tfrac1{8} \, \ddot u_k^2
        + \tfrac1{6} \, \dot u_k \, u^{(3)}_k
        - \tfrac1{8} \, (\partial_x \dot u_k)^2
        - \tfrac1{4} \, \partial_x u_k \, \partial_x \ddot u_k
    \notag \\
  & \qquad \vphantom\int
        + \tfrac1{4} \, f(u_k) \, \ddot u_k
        + \tfrac1{8} \, f' (u_k) \, \dot u_k^2
      \biggr] \, \d x + O(h^4) \,.
  \label{e.d.lagr.approx}
\end{align}

On the other hand, we may substitute the expansions \eqref{e.taylor}
into the continuum action functional and collect terms containing identical
powers of $h$.  We find that
\begin{align}
  S & = \int_0^T L(u,\dot u) \, \d t
        \notag \\
    & = \sum_{k=0}^{n-1} \int_0^h \int_{\T}
        \biggl[
          \tfrac1{2} \,
          (\dot u_k + \ddot u_k \, t + \tfrac1{2} \, u_k^{(3)} \, t^2)^2
          - \tfrac1{2} \,
            (\partial_x u_k + \partial_x \dot u_k \, t
            + \tfrac1{2} \, \partial_x \ddot u_k \, {t^2})^2
        \notag \\
    & \qquad
          + V(u_k) + f(u_k) \, \dot u_k \, t
          + \tfrac12 \, f'(u_k) \, \dot u_k^2 \, t^2
          + \tfrac12 \, f(u_k) \, \ddot u_k \, t^2
        \biggr] \, \d x \, \d t
        + O(h^3)
        \notag \\
    & = \sum_{k=0}^{n-1}
        \biggl[
          h \, L(u_k, \dot u_k)
          + \frac{h^2}{2} \int_{\T}
            \dot u_k \, \ddot u_k -
            \partial_x u_k \, \partial_x \dot u_k +
            f(u_k) \, \dot u_k \, \d x
        \notag \\
    & \qquad + \frac{h^3}{6} \int_{\T}
            \ddot u_k^2 + \dot u_k \, u_k^{(3)} -
            (\partial_x \dot u_k)^2 -
            \partial_x u_k \, \partial_x \ddot u_k +
            f'(u_k) \, \dot u_k^2 + f(u_k) \, \ddot u_k \, \d x
        \biggr]
        + O(h^3) \,.
  \label{e.action.approx}
\end{align}
Inserting the expanded discrete Lagrangian \eqref{e.d.lagr.approx}
into the discrete action \eqref{e.d.action} and comparing with
\eqref{e.action.approx}, we find that
\begin{align}
  \S & = S - \frac{h^3}{24} \sum_{k=0}^{n-1} \int_{\T}
             \ddot u_k^2 - (\partial_x \dot u_k)^2
             + 2 \, \partial_x u_k \, \partial_x \ddot u_k
             + f'(u_k) \, \dot u_k^2
             - 2 \, f(u_k) \, \ddot u_k \, \d x + O(h^3)
       \notag \\
     & = S - \frac{h^2}{24} \int_0^T \int_{\T}
             \ddot u^2 - (\partial_x \dot u)^2
             + 2 \, \partial_x u \, \partial_x \ddot u
             + f'(u) \, \dot u^2
             - 2 \, f(u) \, \ddot u \, \d x \, \d t + O(h^3) \,.
\end{align}
Integrating by parts with respect to time, we obtain---up to boundary
terms which do not contribute to the variational principle because $u$
and $\dot u$ are held fixed at the temporal endpoints---that
\begin{equation}
  \S = \int_0^T
  \biggl[
    L(u,\dot u) - h^2 \int_{\T}
    \tfrac{1}{24} \, \ddot u^2 - \tfrac18 \, (\partial_x \dot u)^2
    + \tfrac18 \, f'(u) \, \dot u^2 \, \d x
  \biggr] \, \d t + O(h^3) \,.
  \label{e.comparison}
\end{equation}
Thus, the discrete variational principle with Lagrangian $\L$ is
equivalent---up to terms of $O(h^3)$---to the continuous variational
principle for the modified Lagrangian
\begin{equation}
  \label{e.mod-lagr}
  L_\modi (u, \dot u, \ddot u; h)
  = L(u, \dot u ) - h^2 \int_{\T}
    \tfrac{1}{24} \, \ddot u^2 - \tfrac18 \, (\partial_x \dot u)^2
    + \tfrac18 \, f'(u) \, \dot u^2 \, \d x \,.
\end{equation}

We now seek stationary points of the modified action with respect to
variation where $u$ and $\dot u$ are held fixed at $t=0$ and $t=T$.
The resulting Euler--Lagrange equations read, abstractly,
\begin{equation}
  \label{e.EL.uddot}
  \frac{\d^2}{\d t^2} \frac{\delta L_\modi}{\delta \ddot u}
  - \frac{\d}{\d t} \frac{\delta L_\modi}{\delta \dot u}
  + \frac{\delta L_\modi}{\delta u} = 0 \,.
\end{equation}
This expression evaluates to
\begin{equation}
  \ddot u - \partial_{xx} u - f(u)
  + h^2 \, \bigl[
             \tfrac1{12} \, u^{(4)}
             - \tfrac14 \, \partial_{xx} \ddot u
             - \tfrac18 \, f''(u) \, \dot u^2
             - \tfrac14 \, f'(u) \, \ddot u
           \bigr]
  = 0 \,.
\end{equation}
We note that this equation is of higher order with respect to time
compared to the original semilinear wave equation, i.e., it has the
form of a singular perturbation problem in a bigger phase space.  As
this is undesirable, we seek to restrict the modified dynamics to the
phase space of the original equation.

In this simple example, this could be done \emph{ad hoc} by using the
second time derivative of the semilinear wave equation to eliminate
the fourth time derivative from the modified equation---an
approximation which gives a formally correct result.  In this case,
the resulting equation would only contain frequencies on the order of
those already present in the original problem and, although
non-variational, approximately preserve energy over long times.  In
the next section, we show that this can also be addressed by a
systematic variational construction.

\section{Method of degenerate variational asymptotics}
\label{s.deg}

We introduce a near-identity configuration space transformation of the
general form
\begin{equation}
  \label{e.transform}
  u_h = u + h^2 \, w \,,
\end{equation}
where $u_h$ denotes the solution curve in old physical configuration
space coordinates and $u$ denotes the solution in a new coordinate
system in which the modified equation will be computed.  The field $w$
can be seen as the leading order generator of the transformation.  It
will be chosen \emph{a posteriori} in such a way that the transformed
modified Lagrangian \eqref{e.mod-lagr}, when truncated to $O(h^2)$,
will not depend on time derivatives of order two and higher.  We
compute
\begin{align}
  L_\modi (u, \dot u, \ddot u; h)
  & = \int_{\T} \tfrac1{2} \,
        (\dot u + h^2 \, \dot w)^2 -
        \tfrac1{2} \, (\partial_x u + h^2\, \partial_x w)^2 +
        V(u + h^2 \, w) \, \d x
      \notag \\
  & \quad - h^2
      \int_{\T} \tfrac{1}{24} \, \ddot u^2 -
        \tfrac18 \, (\partial_x \dot u)^2 +
        \tfrac18 \, f'(u) \, \dot u^2 \, \d x + O(h^3)
      \notag \\
  & = L(u,\dot u) - h^2 \int_{\T}
        \ddot u \, w + \partial_x u \, \partial_x w
        - f(u) \, w \, \d x
      + h^2 \, \frac{\d}{\d t} \int_{\T} \dot u \, w \, \d x
      \notag \\
  & \quad - h^2
      \int_{\T} \tfrac{1}{24} \, \ddot u^2 -
        \tfrac18 \, (\partial_x \dot u)^2 +
        \tfrac18 \, f'(u) \, \dot u^2 \, \d x
      + O(h^3) \,.
  \label{e.lagr.trans}
\end{align}
We now observe that choosing $w=- \tfrac1{24} \, \ddot u+G(u)$, where
$G$ is some functional acting on $u(\,\cdot\,,t)$ for $t$ fixed, we
formally eliminate the excess time derivatives from the $O(h^2)$ terms
in \eqref{e.lagr.trans}. The choice of $G$ is, in principle, entirely
free.  It appears sensible, though, to consider only choices that are
dimensionally consistent.  Specializing further, two cases stand out.
The apparently simplest choice is $G=0$, so that
\begin{equation}
  w = - \tfrac1{24} \, \ddot u  \,.
  \label{e.transformation}
\end{equation}
Alternatively, we might take
\begin{equation}
  w = - \tfrac1{24} \, \ddot u + \tfrac1{24} \partial_{xx} u
      + \tfrac1{24} f(u) \,,
  \label{e.transformation1}
\end{equation}
in which case $w$ vanishes formally up to terms of $O(h^2)$.

Inserting \eqref{e.transformation} into \eqref{e.lagr.trans},
integrating by parts, discarding all perfect time derivatives, and
truncating to $O(h^2)$, we obtain the modified Lagrangian
\begin{equation}
  L_\modi (u, \dot u; h)
    = L(u,\dot u) +
      h^2 \int_{\T}
      \biggl[
   \tfrac1{12}  \, (\partial_x \dot u)^2 -
        \tfrac1{12}  \, f'(u) \, \dot u^2
      \biggr] \, \d x \,.
\end{equation}
We compute
\begin{subequations}
  \label{e.La.partial}
\begin{align}
  \frac{\delta L_\modi}{\delta \dot u}
  & = \bigl(
        1 - \tfrac{h^2}{6} \, f'(u) -
        \tfrac{h^2}{6} \, \partial_{xx}
      \bigr) \, \dot u \\
\intertext{and}
  \frac{\delta L_\modi}{\delta u} & =  \partial_{xx} u + f(u)
   - \tfrac{h^2}{12} \,
       f''(u) \, \dot u^2 \,.
\end{align}
\end{subequations}
Hence, the transformed modified Euler--Lagrange equations read
\begin{equation}
  \bigl(
    1 - \tfrac{h^2}{6} \, f'(u) -
    \tfrac{h^2}{6} \, \partial_{xx} \bigr) \, \ddot u
  - \partial_{xx} u - f(u) - \tfrac{h^2}{12} \,
     \, f''(u) \, \dot u^2 = 0 \,.
  \label{e.modified-el}
\end{equation}
The equations of motion conserve the energy
\begin{align}
  H_\modi
  & = \bigg\langle
        \frac{\delta L_\modi}{\delta \dot u} , \dot u
      \bigg\rangle - L_\modi
      \notag \\
  & \begin{aligned}
      = H + \frac{h^2}{12} \int_{\T}
      \biggl[
      &  (\partial_x \dot u)^2
         -  f'(u) \, \dot u^2
      \biggr] \, \d x \,,
    \end{aligned}
  \label{e.hmod}
\end{align}
where $\langle \cdot \, , \cdot \rangle$ denotes the $L^2$-pairing
between $Q$ and $Q^*$.

We remark that the transformation we use as well as the resulting
Hamiltonian are different from classic Hamiltonian normal form theory
(see, e.g., \cite{Arnold:1988:GeometricalMT} for a general exposition
and \cite{CotterR:2006:SemigeostrophicPM} for an application of normal
form theory to a fast-slow system with gyroscopic forcing for which
the method of degenerate variational asymptotics was originally
developed).  To see this, one can write the modified Hamiltonian
\eqref{e.hmod} in canonical variables $u$ and
$p = \delta L/ \delta \dot{u}$, expand in powers of $h$ so that
$H_\modi = H_0 + h^2 \, H_2 + O(h^4)$ and verify that the Poisson
bracket $\{H_0,H_2\}$ does not vanish.  The construction principles
behind the two approaches are also markedly different: the solution of
the modified system explicitly determines our transformation via
\eqref{e.transform} and \eqref{e.transformation}, whereas normal form
transformation is obtained from modified Hamiltonian by solving the
homological equation.

\section{Initialization}
\label{s.init}

A variational derivation of the modified equation will naturally yield
a second order system which may be cast into a system of first order
equations in several equivalent ways. Hence, care must be taken in
matching of the initial and final time data of the Lagrangian modified
equation with that of the implicit midpoint numerical scheme. In other words, a momentum variable $p$ of the first order system equivalent to the modified equation can be chosen in many different ways and the initialization procedure should be consistent with the chosen definition of $p$.

A straightforward Taylor expansion of the discrete Legendre transform
\eqref{e.d.legendre} shows that
\begin{align}
  p_k & = \dot u_k + \tfrac{h}2 \,
          \bigl(
            \ddot u_k - \partial_{xx} u_k - f(u_k)
          \bigr) + h^2 \,
          \bigl(
            \tfrac16 \, u_k^{(3)} - \tfrac14 \, \partial_{xx} \dot u_k
            - \tfrac14 \, f'(u_k) \, \dot u_k
          \bigr) + O(h^3)
          \notag \\
      & = \dot u_k - \tfrac{h^2}{12} \,
          \bigl(
            \partial_{xx} \dot u_k + f'(u_k) \, \dot u_k
          \bigr) + O(h^3) \,.
  \label{e.udot-to-p}
\end{align}
This relation directly implies that
\begin{equation}
  \dot u_k = p_k + \tfrac{h^2}{12} \,
             \bigl(
               \partial_{xx} p_k + f'(u_k) \, p_k
             \bigr) + O(h^3) \,.
  \label{e.p-to-udot}
\end{equation}
Note that this expression coincides with the first equation
\eqref{e.PDE.mod1a} of the Hamiltonian modified system.  It implies
that the initial data for $p$ cannot be identified with the initial
data for $\dot u$ to the order that the modified equation is valid.
Rather, if $p$ denotes the initial data for the implicit midpoint
rule, then the initialization of $\dot u$ for the modified equation
needs to receive data which is related to $p$ via \eqref{e.p-to-udot}.
Vice versa, the final time data for $\dot u$ of the modified equation
needs to be expressed in terms of the implicit midpoint $p$ via
\eqref{e.udot-to-p} before the two can be consistently compared.

We remark that this relation must be used when consistently comparing
any modified equation which is second order in time, variational or
not, to the implicit midpoint rule.  The change of coordinates which
appears in our method of degenerate variational asymptotics appears
additionally when investigating the variational modified equations
numerically, as explained next.

\section{Numerical Experiments}
\label{s.num}

The modified equations derived in
Section~\ref{s.deg} takes form
\begin{equation}
  K(u) \ddot u - \partial_{xx} u - f(u)
  - \tfrac{h^2}{12} \, f''(u) \, \dot u^2 = 0
  \label{e.mod-ab0}
\end{equation}
with
\begin{equation}
  K(u) v = \bigl(
          1 - \tfrac{h^2}6 \, f'(u) -  \tfrac{h^2}6 \, \partial_{xx}
        \bigr) v \,.
  \label{e.K}
\end{equation}
We choose $p=\dot u$ as for the original semilinear wave equation.
(We could equally well choose $p = \delta L_\modi/\delta \dot u = K(u)
\dot u$ in which case the initialization described in
Section~\ref{s.init} has to be consistently adapted.)  Then the
corresponding first order system reads
\begin{equation}
  \dot U \equiv
  \begin{pmatrix}
    \dot u \\ \dot p
  \end{pmatrix}
  =
  \begin{pmatrix}
    p \\
    K^{-1} \bigl(
             \partial_{xx} u + f(u) +
             \tfrac{h^2}{12} \, f''(u) \, p^2
           \bigr)
  \end{pmatrix} \,,
  \label{e.corresponding1st}
\end{equation}
where the solution of operator equation $K(u)v = z$ is computed as the
fixed point of the contraction mapping
\begin{equation}
  v = (1 - \tfrac{h^2}6 \, \partial_{xx})^{-1}
      (z + \tfrac{h^2}6 \, f'(u)\, v) \,.
  \label{e.fixedpoint}
\end{equation}
We note that the linear modified dynamics has Fourier representation
\begin{equation}
  \partial_t \hat U_{\waven}
  =
  \begin{pmatrix}
    0 & 1 \\ \tfrac{-{\waven}^2}{1 + \frac16 \, h^2 \, \waven^2} & 0
  \end{pmatrix} \, \hat U_{\waven}k \,.
  \label{e.lin-var}
\end{equation}
Clearly, the highest frequencies are $O(h^{-1})$ without restrictions
on the spatial wave numbers, which is markedly different from the
classic linear modified equations \eqref{e.lin-ham}.

\begin{figure}[tb]
\centering
\includegraphics{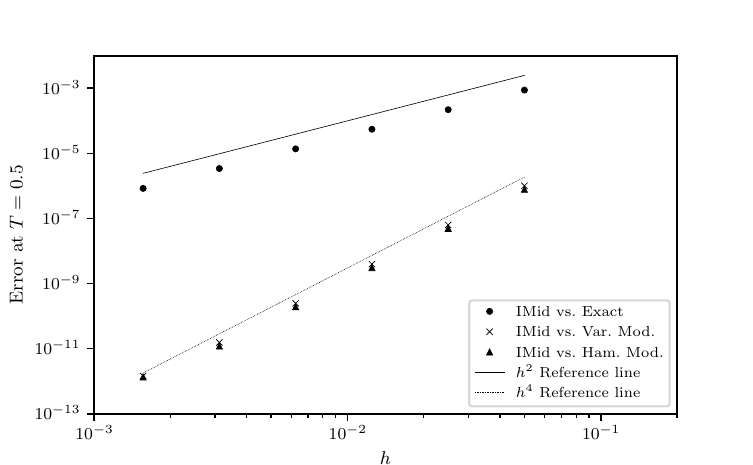}
\caption{Scaling of the error at final time $T=0.5$.}
\label{f.scaling}
\end{figure}

\begin{figure}[tb]
\centering
\includegraphics{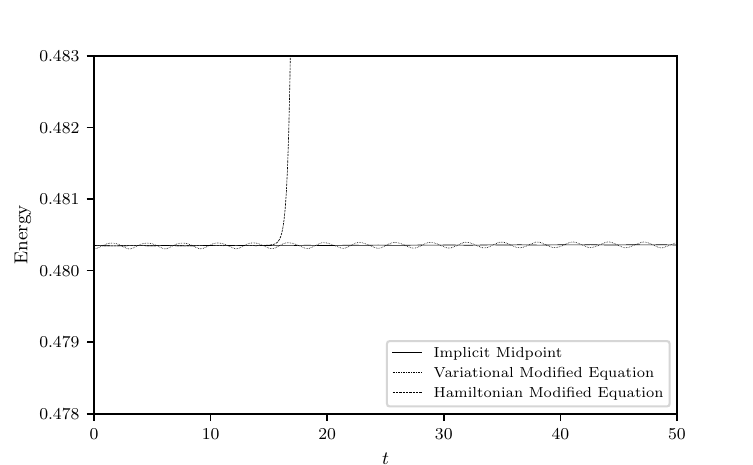}
\caption{Approximate energy preservation of the implicit midpoint rule
and modified equations.  The Hamiltonian modified equation has a
numerical blowup due to a CFL violation while the variational modified
equations is stably solved under identical conditions using a fourth
order explicit Runge--Kutta scheme with a fixed step size $\Delta
t=0.025<h=0.037$.}
\label{f.energy}
\end{figure}

When initializing the modified equation, we must invert the
transformation \eqref{e.transformation}, which now takes the form
\begin{equation}
  U_h = U - \tfrac1{24} \, h^2 \, \ddot U \,.
  \label{e.transformation2}
\end{equation}
This relation can also be cast into fixed point form as follows.
Differentiating \eqref{e.mod-ab0} in time, we find
\begin{equation}
  K u^{(3)} - \partial_{xx} \dot u - f'(u) \, \dot u
  - \tfrac{h^2}{12} \, f'''(u) \, \dot u^3
  - \tfrac{h^2}3 \, f''(u) \, \dot u \, \ddot u = 0 \,.
\end{equation}
Then, for a given vector $U_h$, \eqref{e.transformation2} can be
solved for $U$ via the contraction mapping
\begin{subequations}
\begin{gather}
  U = U_h + \tfrac1{24} \, h^2 \, \ddot U \,, \\
  \ddot U = (1 - \tfrac{h^2}6 \, \partial_{xx})^{-1} G(U, \ddot U) \,,
\end{gather}
\end{subequations}
with
\begin{equation}
  G(U, \ddot U) =
  \begin{pmatrix}
    \frac{h^2}6 \, f'(u) \, \ddot u + \partial_{xx} u + f(u)
      + \tfrac{h^2}{12} \, f''(u) \, p^2 \\
    \frac{h^2}6 \, f'(u) \, \ddot p + \partial_{xx} p
      + f'(u) \, p + \tfrac{h^2}{12} \, f'''(u) \, p^3
      + \tfrac{h^2}3 \, f''(u) \, p \, \ddot u
  \end{pmatrix} \,.
\end{equation}

In our numerical example, we take the nonlinearity $V(u) =
-\tfrac1{10} \, u^4$.  Figure~\ref{f.scaling} shows that both the
Hamiltonian and the variational modified system show the expected
$O(h^4)$ scaling when compared to the solution of the implicit
midpoint rule.  (Due to the symmetry of the method only even powers of
$h$ appear in the modified equations at any order.)  The modified
equations were solved with a highly resolved standard forth order
explicit Runge--Kutta method.

Figure~\ref{f.energy} shows the approximate preservation of energy of
the implicit midpoint rule and of the two modified systems.  The
occurrence of high frequencies in the Hamiltonian construction leads
to numerical blowup, unless the stepsize used in the explicit
Runge--Kutta scheme is adjusted to fit a stricter CFL bound.  The
graphs display the semilinear wave energy; the respective modified
energies would be exactly preserved by the true solutions to the
modified equations.

\section{Well-posedness of the new modified system}
\label{s.wellposed}

In the following we provide a functional framework which shows that,
in the limit of vanishing stepsize $h$, the variational modified
system with the choice of parameters as in Section~\ref{s.num} behaves
analytically like the original semilinear wave equation.  We begin by
proving a statement on the operator $K$ defined in \eqref{e.K}.  For
convenience, we write $K(u) = L+M(u)$ with
\begin{equation}
  Lv = \bigl(
          1  - \tfrac{h^2}6 \, \partial_{xx}
       \bigr) v
\end{equation}
and
\begin{equation}
  M(u)v = - \tfrac{h^2}6 \, f'(u) \, v \,.
\end{equation}

Let $L^2(\T)$ denote the Lebesgue space of square integrable functions
on the circle, which we endow with norm
\begin{equation}
  \norm{v}{L^2} = \sum_{k \in Z} \lvert v_k \rvert^2 \,,
\end{equation}
$L^\infty(\T)$ the space of essentially bounded functions endowed with
the usual essential sup-norm, and $H^s(\T)$ the Sobolev space of
functions whose generalized derivative of order $s$ belongs to
$L^2(\T)$, endowed with norm
\begin{equation}
  \norm{v}{H^s}
  = \sum_{\waven \in Z} (1 + \lvert \waven \rvert^{2s}) \,
    \lvert v_{\waven} \rvert^2 \,.
\end{equation}

\begin{lemma} \label{l.K}
For every $C>0$ there exists $h_*>0$ such that for every $h \in
(0,h_*]$, $u \in H^1(\T)$ with $\tnorm{u}{H^1} \leq C$, and $z \in
L^2(\T)$, the equation $K(u)v=z$ has a unique solution $v \in H^2(\T)$
and there exists a constant $c=c(C)>0$ such that
\begin{equation}
  \label{e.estimate1}
  \norm{v}{L^2} \leq \frac{1}{1 - c \, h^2} \, \norm{z}{L^2}\,,
\end{equation}
and
\begin{equation}
  \label{e.estimate2}
  \norm{v}{H^2} \leq
  6 \, \biggl(
         \frac1{h^2}
         + \frac{c}{1-c \, h^2}
       \biggr) \, \norm{z}{L^2} \,.
\end{equation}
Moreover, for fixed $z \in L^2(\T)$ and under the above bounds on $u$
and $h$, the mapping $u \mapsto K(u)^{-1}z$ is uniformly Lipschitz
continuous as a map from $H^1$ to $H^2$.
\end{lemma}

\begin{proof}
As in \eqref{e.fixedpoint}, we write $Kv=z$ in fixed point form as
\begin{equation}
  v = F(v) \equiv L^{-1}(z-M(u)v) \,.
  \label{e.fixedpoint2}
\end{equation}
Since $L^{-1}$, the inverse Helmholtz operator, has norm $1$ as an
operator from $H^s$ to $H^s$, there exists a constant $c>0$ such that
\begin{equation}
  \norm{F(v_1) - F(v_2)}{L^2}
  \leq c \, \tfrac{h^2}6 \, \norm{f'(u)}{H^1} \, \norm{v_1-v_2}{L^2} \,.
\end{equation}
Hence, there exists $h_*>0$ such that for all $h \in [0,h_*]$ the map
$F$ is a contraction, hence has a unique fixed point $v$ by the
contraction mapping theorem.

Taking the $L^2$-norm of \eqref{e.fixedpoint2}, we obtain
\begin{equation}
  \norm[2]{v}{L^2}
  \leq \norm{z}{L^2}
  + \tfrac{h^2}6 \, \norm{f'(u)}{L^\infty} \, \norm{v}{L^2} \,,
\end{equation}
which implies \eqref{e.estimate1}.  Taking the $H^2$-norm of
\eqref{e.fixedpoint2} and noting that $L^{-1}$ has norm $6/h^2$ as an
operator from $L^2$ into $H^2$, we find
\begin{equation}
  \norm[2]{v}{H^2}
  \leq \tfrac6{h^2} \, \norm{z}{L^2}
       + \norm{f'(u)}{L^\infty} \, \norm{v}{L^2}
\end{equation}
which together with \eqref{e.estimate1} implies \eqref{e.estimate2}.

Now suppose that $K(u_1)v_1=z$ and $K(u_2)v_2=z$.  Then
\begin{equation}
  v_1 - v_2
  = - L^{-1} \bigl(
               M(u_1) \, (v_1-v_2) +
               (M(u_1) - M(u_2)) \, v_2
             \bigr) \,.
  \label{e.difference}
\end{equation}
Taking the $L^2$-norm on both sides, we find that
\begin{equation}
  \norm{v_1-v_2}{L^2}
  \leq \tfrac{h^2}6 \, \norm{f'(u_1)}{L^\infty} \, \norm{v_1-v_2}{L^2}
  + \tfrac{h^2}6 \, \norm{f'(u_1)-f'(u_2)}{L^\infty} \,
    \norm{v_2}{L^2} \,,
\end{equation}
so that, for some $\tilde c = \tilde c(C)$,
\begin{equation}
  \norm{v_1-v_2}{L^2}
  \leq \frac{\tilde c \, h^2}{1-c \, h^2} \,
       \norm{v_2}{L^2} \, \norm{u_1-u_2}{H^1} \,.
\end{equation}
Due to \eqref{e.estimate1}, this estimate implies uniform Lipschitz
continuity of $u \mapsto K(u)^{-1}z$ as a map from $H^1$ into $L^2$.
Then, taking the $H^2$-norm of \eqref{e.difference} and using the
operator norm of $L^{-1}$ as a map from $L^2$ to $H^2$ implies uniform
Lipschitz continuity into $H^2$ as well.
\end{proof}

To proceed, we write $H^1_h(\T)$ to denote the space $H^1(\T)$ endowed
with the nonuniform norm
\begin{equation}
  \norm[2]{v}{H^1_h}
  = \sum_{k \in Z} (1 + \tfrac{h^2}6 \, \lvert k \rvert^{2}) \,
    \lvert v_k \rvert^2 \,.
\end{equation}
Note that the space $H^1 \times H_h^1$ is the ``energy space'' which
corresponds to the modified Hamiltonian \eqref{e.hmod} with $a=b=0$.
We therefore seek local well-posedness in this space.

\begin{theorem}
For every $C>0$ there exists $T=T(C)>0$ and $h_*=h_*(C)>0$ such that
for every $U_0 \in H^1(\T) \times H^1_h(\T)$ with $\tnorm{U_0}{H^1
\times H_h^1} \leq C$ and every $h \in (0,h_*]$ there exists a unique
mild solution $U \in C([0,T]; H^1(\T) \times H^1_h(\T)$ to the new
modified equation \eqref{e.wave2} with bounds which remain uniform in
$h$.
\end{theorem}

\begin{proof}
We first observe, as can be checked by direct computation, that
\begin{equation}
  K^{-1} - L^{-1} = -K^{-1} M L^{-1} \,,
\end{equation}
so that the new modified system \eqref{e.corresponding1st} can be
written as a semilinear evolution equation of the form \eqref{e.wave2}
with
\begin{equation}
  A =
  \begin{pmatrix}
    0 & 1 \\ L^{-1} \partial_{xx} & 0
  \end{pmatrix}
\end{equation}
and
\begin{equation}
  B(U) \equiv
  \begin{pmatrix}
  0 \\ b(U)
  \end{pmatrix}
  =
  \begin{pmatrix}
    0 \\ K^{-1}
    \bigl(
      -ML^{-1} \partial_{xx} u + f(u)
      + \tfrac{h^2}{12} \, f''(u) \, p^2
    \bigr)
  \end{pmatrix} \,.
\end{equation}

Let $H^1 \times H_h^1/\R^2$ denotes the space of $H^1 \times H_h^1$ functions with vanishing mean endowed with the norm
\begin{equation}
  \norm[2]{U}{H^1 \times H_h^1/\R^2}
  = \norm[2]{\partial_x u}{L^2} + \norm[2]{p}{H_h^1} \,.
\end{equation}

The crucial observation is that $A$ generates a unitary group $\exp(tA)$ on $H^1 \times H_h^1/\R^2$. Moreover,
\begin{align}
  b(U_1) - b(U_2)
  & = \bigl(
        K(u_1)^{-1} - K(u_2)^{-1}
      \bigr) \,
      \bigl(
        -M(u_1) \, L^{-1} \partial_{xx} u_1 + f(u_1)
        + \tfrac{h^2}{12} \, f''(u_1) \, p_1^2
      \bigr)
    \notag \\
  & \quad
    - K(u_2)^{-1} \,
      \bigl(
        M(u_1) - M(u_2)
      \bigr) \, L^{-1} \partial_{xx} u_1
    \notag \\
  & \quad
    + K(u_2)^{-1} \,
      \bigl[
        -M(u_2) \, L^{-1} \partial_{xx} (u_1 - u_2) + f(u_1) - f(u_2)
    \notag \\
  & \qquad
        + \tfrac{h^2}{12} \,
          \bigl(
            (f''(u_1) - f''(u_2)) \, p_1^2
            + f''(u_2) \, (p_1 + p_2) \, (p_1 - p_2)
          \bigr)
      \bigr] \,.
  \label{e.bdifference}
\end{align}
Taking the $L^2$-norm of this expression, using the uniform Lipschitz
continuity of $K^{-1}$ and estimate \eqref{e.estimate1} as asserted by
Lemma~\ref{l.K}, and supposing that $U$ has an $H^1 \times H_h^1$
bound of $2C$, say, we find that there exists $h_*>0$ such that for
all $h \in (0,h_*]$,
\begin{align}
  \norm{b(U_1) & - b(U_2)}{L^2}
  \leq c_1 \, \norm{u_1-u_2}{H^1} \notag \\
  & \qquad \qquad \cdot
         \bigl(
           \norm{M(u_1) \, L^{-1} \partial_{xx} u_1}{L^2} +
           \norm{f(u_1)}{L^2} +
           h^2 \, \norm{f''(u_1) \, p_1^2}{L^2}
         \bigr) \notag \\
  & \quad
    + c_2 \, h^2 \, \norm{f'(u_1) - f'(u_2)}{L^\infty} \,
                    \norm{L^{-1} \partial_{xx} u_1}{L^2}
    \notag \\
  & \quad
    + c_3 \,
         \bigl[
           \norm{M(u_2) \, L^{-1} \partial_{xx} (u_1-u_2)}{L^2} +
           \norm{f(u_1) - f(u_2)}{L^2} +
           h^2 \,
           \bigl(
             \norm{p_1}{L^\infty} + \norm{p_2}{L^\infty}
           \bigr) \notag \\
  & \qquad\qquad \cdot
           \bigl(
              \norm{f''(u_1) - f''(u_2))}{L^\infty} \,
              \norm{p_1}{L^2}
              + \norm{f''(u_2)}{L^\infty} \,
              \norm{p_1-p_2}{L^2}
           \bigr)
         \bigr]
\end{align}
Noting that
\begin{equation}
  h^2 \, \norm{L^{-1} \partial_{xx} v}{L^2}
  \leq c_3 \, \norm{v}{L^2}
\end{equation}
and
\begin{equation}
  \norm{p}{L^\infty} \leq c_4 \, h^{-2} \, \norm{p}{H_h^1}
\end{equation}
uniformly under the assumed bounds on $h$ and $U$, we find that
\begin{equation}
  \norm{b(U_1) - b(U_2)}{L^2}
  \leq c_5 \, \norm{U_1-U_2}{H^1 \times H_h^1} \,.
\end{equation}
A similar argument, now taking the $H^1$-norm of \eqref{e.bdifference}
and using \eqref{e.estimate2} rather than \eqref{e.estimate1} shows
that
\begin{equation}
  \norm{b(U_1) - b(U_2)}{H^1}
  \leq c_6 \, h^{-2} \, \norm{U_1-U_2}{H^1 \times H_h^1} \,.
\end{equation}
This proves uniform Lipschitz continuity of $B$ in $H^1 \times H_h^1$.

Altogether, we find that the above formulation of the new modified
equation fits into the standard framework for mild solutions of
semilinear evolution equations \cite{Henry:1981:GeometricTS,
Pazy:1983:SemigroupsLO}, from which the conclusion ensues.
\end{proof}

This theorem shows that the new modified equation is well posed in a
space which limits to the standard setting $H^1 \times L^2$ for the
original semilinear wave equation as $h \to 0$.  We note that the
theorem and proof easily translates up the scale of standard Sobolev
spaces.

\section{All-order modified equations: the linear case}
\label{s.higher-linear}

Let us now consider how the procedure set up in
Sections~\ref{s.bea-lagrange} and~\ref{s.deg} generalizes to higher
orders.  We first consider the linear case, which can be treated for
all orders at once via a generating function approach.  In
Section~\ref{s.higher-nl} below, we extend this approach to the
nonlinear case using an iterative construction that terminates at
fixed, but arbitrary order.

We say that modified Lagrangians $L_1$ and $L_2$ are equivalent at
order $m$ and write
\begin{equation}
L_1 (u_h, \dot u_h, \dots) \sim L_2(u_h, \dot u_h, \dots )  
\end{equation}
whenever the difference between solutions of Euler-Lagrange equations
for $L_1$ and $L_2$ with their respective arguments held fixed at the
temporal endpoints is $o(h^{m})$. Two formal power series are
equivalent whenever they are equivalent at every finite order $m$.

First, by performing all steps laid out in
Section~\ref{s.bea-lagrange} consistently at any order, we find that the
general modified Lagrangian for the implicit midpoint rule applied to
the linear wave equation reads
\begin{equation}
  \label{e.mod-lagr-general} 
  L_\modi^\lin (u_h, \dot u_h, \dots) \sim
   \sum_{i=0}^\infty (-1)^i \, h^{2i}
    \int_{\T} a_i \, (u_h^{(i+1)})^2
      - b_i \, (\partial_x u_h^{(i)})^2 \, \d x
\end{equation}
where $a_0 = b_0 = 1/2$ and, for $i \geq 1$,
\begin{gather}
  a_i = \frac{1}{(2i+2)!}
  \qquad \text{and} \qquad
  b_i = \frac{1}{4 \, (2i)!} \,.
\end{gather}

Proceeding formally, we integrate by parts, recognize the resulting
power series as cosine series, and finally apply standard
trigonometric identities to obtain
\begin{align}
  L_\modi^\lin
  & \sim \frac{1}{h^2} 
      \int_{\T} u_h \, \biggl[
        u_h - \sum_{i=0}^\infty \frac{h^{2i}}{(2i)!} \, u^{(2i)}
        + \frac{h^2}4 \, \partial_x^2 u_h
        + \frac{h^2}4 \sum_{i=0}^\infty
          \frac{h^{2i}}{(2i)!} \, \partial_x^2 u^{(2i)}
      \biggr] \, \d x
      \notag \\
  & = \frac1{h^2}
      \int_{\T} u_h \, \bigl(
                         1 -\cos 2 {T} - (1+\cos2 {T}) \,{X}^2
                       \bigr) u_h \, \d x
      \notag \\
  & = \frac2{h^2}
      \int_{\T} u_h \, \cos^2 {T} \, (\tan^2 {T} -  {X}^2) u_h \, \d x
  \label{e.lmod-untransformed}
\end{align}
where the linear operators ${T}$ and ${X}$ are given by
\begin{equation}
  T = \frac{-\I h \partial_t}2
  \qquad \text{and} \qquad
  X = \frac{-\I h \partial_x}2 \,.
\end{equation}
and $\I$ stands for the imaginary unit. 
Let us now make the transformation ansatz
\begin{equation}\label{e.transansatz}
  u_h = \phi(T^2, X^2) u \,.
\end{equation}
Plugging this ansatz into \eqref{e.lmod-untransformed} and noting that
$T^2$ and $X^2$ are commuting self-adjoint operators, we have
\begin{equation}
  L_\modi^\lin
   \sim \frac2{h^2} \int_{\T} u \,
      \cos^2 T \, (\tan^2 T - X^2) \, \phi^2(T^2, X^2) u \, \d x \,.
  \label{e.lmod-transformed}
\end{equation}
Thus, a sufficient condition for removing time derivatives of order
larger than $2$ from the modified Lagrangian is
\begin{equation}
  \label{e.transm}
  \cos^2 T \, (\tan^2 T - X^2) \, \phi^2(T^2,X^2)
  = \psi^2(X) \, (T^2 - \theta^2(X)) \,,
\end{equation}
where $\theta$ and $\psi$ are analytic functions.  Solving for $\phi$,
we obtain
\begin{equation}
  \label{e.phi}
  \phi(T^2,X^2)
  = \frac{\psi(X)}{\cos T} \,
    \sqrt{\frac{T^2 - \theta^2(X)}{\tan^2 T - X^2}} \,.
\end{equation}
Now the task is to (i) find a function $\theta(x)$, analytic in a
neighborhood of the origin, such that $\phi(t,x)$ is analytic in a
neighborhood of the origin and (ii) find a function $\psi(x)$ such
that $\phi$ is a small perturbation of the identity in the sense that
$\phi(0,x)=1$.  We show that the unique choice
\begin{equation}
  \theta(x) = \arctan x
  \qquad \text{and} \qquad
  \psi(x) = \frac{x}{\arctan x}
  \label{e.theta-psi}
\end{equation}
fulfills these conditions. The proof is based on the following fact.

\begin{lemma}\label{l.an}
Let $f(x)=\sum_{k=0}^\infty f_k \, x^k$ be analytic in a neighborhood
of origin. Then the function
\begin{equation}
  F(s,x) = x \, \frac{f(s)-f(x)}{s-x}
\end{equation}
is analytic near the origin. Moreover, if $f$ is bounded on
$[0, \infty)$, the coefficients of the Maclaurin expansion of $F$ with
respect to $s$ are bounded functions of $x$.
\end{lemma}

\begin{proof}
We compute, changing the order of summation in the last step,
\begin{equation}
  F(s,x)
  = x \sum_{k=0}^\infty f_k \, \frac{s^k-x^k}{s-x} 
  = x \sum_{k=0}^\infty f_{k+1} \sum_{j=0}^k s^{k-j} \, x^j  
  = \sum_{k=0}^\infty F_k(x) \, s^k  \,,
\end{equation}
where the Maclaurin coefficients $F_k(x)$ are given by
\begin{equation}
  F_k(x)
  = \sum_{j=1}^{\infty} f_{k+j} \, x^j 
  = \frac{1}{x^{k}} \sum_{j=1}^{\infty} f_{k+j} \, x^{k+j} 
  = \frac{1}{x^{k}} \,
    \biggl( f(x) - \sum_{j=0}^k f_j \, x^j \biggr) \,.
\end{equation}
Hence, $F$ is analytic.  Moreover, if $f$ is bounded on $[0,\infty)$,
\begin{equation}
  \lim\limits_{x \rightarrow 0} F_k (x) = 0 \,, \qquad 
  \lim\limits_{x \rightarrow \infty} F_k (x) = f_k \,,
\end{equation}
so $F_k$ is also bounded on $[0,\infty)$.
\end{proof}

\begin{corollary} \label{c.1}
Suppose that, in addition, $f(0)=0$, $f'(0) \neq 0$, $f$ has an
analytic inverse near the origin, and $f$ is nonzero on $(0,\infty)$.
Then
\begin{equation}
  G(t,x) = \sqrt{\frac{t-f(x)}{f^{-1}(t)-x} \, \frac{x}{f(x)}}
\end{equation}
is analytic near the origin, $G(0,x)=1$, and the coefficients of the
Maclaurin expansion of $G$ with respect to $t$ are bounded functions
of $x$ on $[0,\infty)$.
\end{corollary}

\begin{proof}
Analyticity is obvious.  To show boundedness of the Maclaurin
coefficients of $G$, note that the Maclaurin coefficients of
\begin{equation} 
  \frac{t-f(x)}{f^{-1}(t)-x} \, x
\end{equation}
are finite linear combinations of the Maclaurin coefficients $F_k$
from Lemma~\ref{l.an}, hence are bounded.  Dividing by $f(x)$ and
taking the root does not change this conclusion.
\end{proof}

Returning back to \eqref{e.theta-psi}, we see that the stated choice
of $\theta$ is necessary to ensure analyticity of $\phi$ in a
neighborhood of the origin, as
\begin{equation}
  \frac{\theta^2(X)-\arctan^2(X)}{\tan^2{T}-X^2}
\end{equation}
has a non-removable singularity on the lines $X=\pm \arctan{T}$ unless
its numerator vanishes.  By Corollary~\ref{c.1}, it is also
sufficient.  The choice of $\psi$ stated in \eqref{e.theta-psi} is
then necessary to ensure that $\phi(0,x)=1$, i.e., that the
transformation is near-identity.  By Corollary~\ref{c.1}, this choice
is also sufficient.  Moreover, it guarantees that the coefficients of
the Maclaurin expansion of $\phi$ with respect to $t$ are bounded
functions of $x$.

Inserting the expression for $\theta$ into \eqref{e.phi} and using
standard trigonometric identities, we conclude that the transformation
operator has the generating function
\begin{equation}
  \label{e.phi2}
  \phi(T^2,X^2)
  = \frac{X}{\arctan X} \, \sqrt{
          \frac{(1 + \tan^2 T)(T^2 - \arctan^2 X)}%
         {\tan^2 T - X^2}} \,,
\end{equation}
which, when expanded and truncated at any finite order in $T$, has
coefficients which are \emph{bounded} operators in space.

Substituting this choice back into \eqref{e.transm} and then into
$L_\modi^\lin$ in transformed variables, equation
\eqref{e.lmod-transformed}, we obtain
\begin{equation}
  L_\modi^\lin
  = \frac2{h^2} \int_{\T} u \, \psi^2(X) \,
      (T^2 - \arctan^2 X) u \, \d x \,,
  \label{e.lmod-transformed-final}
\end{equation}
so that the linear modified Euler--Lagrange equation, computed to all
orders, reads
\begin{equation}
  T^2 u - \arctan^2 X \, u = 0
\end{equation}
or, more explicitly,
\begin{equation}
  \ddot u + A^2 u = 0 \,,
  \label{e.linear-mod-el}
\end{equation}
where $A = 2/h \, \arctan X$ is the pseudodifferential operator
with symbol
\begin{equation}
  a(k) = \frac2h \, \arctan \frac{hk}2 \,.
\end{equation}
The operator $A$ is bounded on $L^2$ with operator norm $O(h^{-1})$.
Note that this expression reproduces the dispersion relation for the
implicit midpoint rule applied to the linear wave equation
\cite[Section~4.2]{BridgesR:2006:NumericalMH} exactly.

\section{High-order modified equations: the nonlinear case}
\label{s.higher-nl}

To construct the modified Lagrangian for the nonlinear system, we
follow the procedure in Section~\ref{s.bea-lagrange} to a fixed order
$2m$.  We do not write out the higher-order terms explicitly, but note
that they potentially contain multiples of all higher-order terms that
appear on the right hand side of a Fa\`a di Bruno expansion of
$V(u(h))$ with respect to $h$, i.e., may contain time derivatives up
to order $2m$.  To eliminate time derivatives of second and higher
order from the modified Lagrangian, we proceed in two steps. In the
first step, we perform the transformation to all orders exactly as
outlined in Section~\ref{s.higher-linear}. In the second step, we
remove these higher-order time derivatives from the variational
principle iteratively, applying an additional near-identity
transformation at each step.  The procedure for doing so is motivated
by the approach in \cite{GottwaldO:2014:SlowDD}, but is more involved
due to the presence of spatial derivatives.

To see how the modified nonlinear term $V_{\modi}(u_h)$ changes under
the transformation \eqref{e.phi2}, we expand $\phi$ with respect to
$t$ and truncate at order $2m$.  This truncated expansion takes the
form
\begin{equation}
  u_h = u + \sum_{i=1}^m h^{2i} \, B_i \, u^{(2i)} 
\end{equation}
where the $B_i$ are self-adjoint spatial operators that implicitly
depend on $h$, but are uniformly bounded on $L^2$ as $h \to 0$.  In
this sense, this series is well-ordered.  Inserting the full $\phi$
into the quadratic terms and the expanded $\phi$ into $V_{\modi}$, we
obtain that $L_\modi \sim L_\modi^\lin + L_\modi^\nl$ with
\begin{gather}
  L_\modi^\lin
  = \frac12 \int_{\T} u \,
      \biggl(
        \partial_{xx} u
        - \frac{X^2}{\arctan^2 X} \ddot u
      \biggr) \, \d x
  \label{e.lmod-transformed-final-2}
\end{gather}
as in \eqref{e.lmod-transformed-final} and, due to the Fa\`a di Bruno
formula and up to terms of $o(h^{2m})$,
\begin{gather}
  L_\modi^\nl
  = V(u) + \sum_{\substack{2 \leq i \leq 2m \\ i \text{ even}}}
      h^{i} \sum_{1 \cdot \alpha_1 + \cdots + i\alpha_i = i}
      V_\alpha(u) \bigl[
        (\partial_t^1 u)^{\otimes \alpha_1}, \dots,
        (\partial_t^i u)^{\otimes \alpha_i}
      \bigr] \,,
  \label{e.lmodinl}
\end{gather}
where $\otimes \beta$ denotes a $\beta$-fold repetition of the
argument and the $V_{\alpha}(u)$ are $\lvert \alpha \rvert$-linear
forms which smoothly depend on $u$ and $h$ and which are uniformly
bounded on $L^2$ as $h \to 0$.  Due to the uniform boundedness of the
$V_\alpha$, we obtain what is effectively a variational principle for
an ODE, albeit with higher-order time derivatives in the nonlinear
contributions.

We now proceed to iteratively eliminate the higher time derivatives
from \eqref{e.lmodinl}. To structure the discussion, let
\begin{gather}
  \V_i = \Span \{
    V_\alpha(u) \bigl[
        (\partial_t^1 u)^{\otimes \alpha_1}, \dots,
        (\partial_t^i u)^{\otimes \alpha_i}
      \bigr] 
    \colon 1 \cdot \alpha_1 + \cdots + i\alpha_i = i
  \} \,,
\end{gather}
with multilinear forms $V_\alpha$ as described above, denote the
vector spaces of functions appearing in the inner sum of
\eqref{e.lmodinl}.  On $\V_i$, we have a natural equivalence of
elements via integration by parts in time or space under the action
integral.  We write $v_1 \sim v_2$ if the functions
$r_1, r_2 \in \V_i$ give the same contribution to the resulting
Euler--Lagrange equations. Note that any $v_i \in \V_i$ is equivalent
to
\begin{gather}
  v_i \sim F(u)[\dot u^{\otimes i}] + \ddot u \, v_{i-2}
  \label{e.sim}
\end{gather}
where $v_{i-2} \in \V_{i-2}$ and $F(u)[\dot u^{\otimes i}] \in \V_i$
depends on $u$ and $\dot u$ only.  Indeed, if any monomial of $v_i$
contains only first or second time derivatives, it is already of the
required form.  If not, there must be a higher-order time derivatives
and we can ``peel off'' time derivatives by repeated integration by
parts until exactly two are left.  When $i=2$, by the same argument,
the second term on the right of \eqref{e.sim} can always taken to be
zero.

Further, let
\begin{gather}
  \W_{k,\ell} = \Span \{
    h^j \, v_i \colon i \leq 2m-k,
    j-i \geq \ell, 2 \leq j \leq 2m, v_i \in \V_i 
  \} \,.
\end{gather}
The first index limits the maximal number of time derivatives
contained in each of the monomials of elements of $\W_{k,\ell}$ and
the second index gives the excess order in $h$ at which time
derivatives occur.  Clearly, $\W_{k,\ell} = \W_{\ell,\ell}$ if $k \leq
\ell$ and
$\W_{k,\ell} \supset \W_{k',\ell'}$ if $k' \geq k$ and $\ell' \geq
\ell$.

We say that two elements from $\W_{k,\ell}$ are equivalent if they are
agree in the sense of equivalence of the $\V_i$, up to terms of
$o(h^{2m})$.  Note that the entire sum on the right of
\eqref{e.lmodinl} is contained in $\W_{0,0}$.  We are now going to
apply a sequence of transformations to the effect that, up to
equivalences, the resulting Lagrangian density depends only on time
derivatives of order one.  We will do so by using \eqref{e.sim} at
each step to split off higher powers of first derivatives from a
remainder, remove second time derivative contributions via a suitably
chosen transformation, and iterate this process until this remainder
is either of class $\W_{2m-2,0}$ or of class $\W_{0,2m-2}$.  In both
of these terminal cases, the monomials making up the remainder
contribution contain at most two time derivatives, which can always be
integrated by parts such that only first time derivatives remain.

To simplify language and notation, we will use the same symbols for
quantities in old and new variables.  Thus, the transformations will
be referred to as ``replacements,'' which is algebraically equivalent.
As we implement these replacements, a corresponding concatenation of
transformations could be constructed.  However, we will not write them
down explicitly and only note that these transformation only contain
uniformly bounded spatial operators.  To begin we note that, up to
equivalences, the following relations hold.
\begin{enumerate}[label={\upshape(\roman*)}]
\item\label{i} If $w \in \W_{k,\ell}$, then
$\ddot w \in \W_{k-2,\ell-2}$.
\item\label{ii} If $w \in \W_{k,\ell}$, then
$\psi^{-2}(X) w \in \W_{k,\ell}$.
\item\label{iii} If $w \in \W_{k,\ell}$, then
$\psi^{-2}(X) \partial_{xx}w \in \W_{k,\ell-2}$.
\item\label{iv} If $v \in \W_{k,\ell}$ and $w \in \W_{k',\ell'}$, then
$vw \in \W_{0,\ell+\ell'}$.
\item\label{v} If $v \in \W_{k,\ell}$, then the replacement
$u \mapsto u + v$ in the expression for $V$ corresponds to the
replacement $V \mapsto V + w$ for some $w \in \W_{0, \ell}$.
\item\label{vi} If $F \in \W_{k,\ell}$ and $v \in \W_{k',\ell'}$ with
$\ell \leq \ell'$, then the replacement $u \mapsto u + v$ in the
expression for $F$ corresponds to the replacement $F \mapsto F + w$
for some $w \in \W_{0, \ell+\ell'}$.
\end{enumerate}

With these provisions, the modified Lagrangian density, after applying
the all-order linear transformation and \eqref{e.sim}, can be written
\begin{align}
  L_{0,0}
  & = \tfrac12 \, u \, u_{xx}
        - \tfrac12 \, u \, \psi^2(X) \ddot u
        + V(u) + R_{0,0}
      \notag \\
  & \sim \tfrac12 \, u \, u_{xx}
        - \tfrac12 \, u \, \psi^2(X) \ddot u
        + V(u) + V_{0,0} + \ddot u \, v_{2,2}
\end{align}
where $R_{0,0} \in \W_{0,0}$ and, by \eqref{e.sim},
$V_{0,0} \in \W_{0,0}$ can be chosen to depend only on $u$ and
$\dot u$ with $v_{2,2} \in \W_{2,2}$.  Now consider the replacement
\begin{equation}
  u \mapsto u + \psi^{-2}(X) v_{2,2} \,.
\end{equation}
Then
\begin{align}
   \tfrac12 \, u \, u_{xx}
   & \mapsto \tfrac12 \, u \, u_{xx}
       + u \, \psi^{-2}(X) \partial_{xx} v_{2,2}
       + \tfrac12 \, \psi^{-2}(X) v_{2,2} \,
         \psi^{-2}(X) \partial_{xx} v_{2,2}
       \notag \\
   & = \tfrac12 \, u \, u_{xx} + w_{2,0} + w_{0,2} 
\end{align}
for some $w_{2,0} \in \W_{2, 0}$ and $w_{0,2} \in \W_{0,2}$.  Next,
\begin{align}
 - \tfrac12 \, u \, \psi^2(X) \ddot u +  \ddot u \, v_{2,2} 
  & \mapsto - \tfrac12 \, u \, \psi^2(X) \ddot u
      + \tilde w_{0,2}
\end{align}
with $\tilde w_{0,2} \in \W_{0,2}$.  Finally, when transforming $V$
and higher-order Lagrangian densities that contain no more than first
derivatives of $u$, then by \ref{v} and \ref{vi}, the additional
remainder term is also of class $\W_{0,2}$.  Thus, altogether,
the first elimination step results in the replacement
\begin{equation}
  L_{0,0} \mapsto L_{2,2} \equiv \tfrac12 \, u \, u_{xx}
        - \tfrac12 \, u \, \psi^2(X) \ddot u
        + V_{2,2} + R_{2,0} + R_{0,2}^2
\end{equation}
with $V_{2,2} \equiv V + V_{0,0}$, and where $R_{2,0}^2 \in \W_{2,0}$
and $R_{0,2} \in \W_{0,2}$.

Now, for the general case, suppose that we are at a stage where the
modified Lagrangian has been transformed into
\begin{align}
  L_{k,\ell}
  & = \tfrac12 \, u \, u_{xx}
        - \tfrac12 \, u \, \psi^2(X) \ddot u
        + V_{k,\ell} + R_{k,\ell-2} + R_{0,\ell}^k
      \notag \\
  & \sim \tfrac12 \, u \, u_{xx}
        - \tfrac12 \, u \, \psi^2(X) \ddot u
        + V_{k,\ell} + F_{k,\ell-2}
        + \ddot u \, v_{k+2,\ell} + R_{0,\ell}^k
\end{align}
where, $R_{k,\ell-2} \in \W_{k,\ell-2}$ and by \eqref{e.sim},
$F_{k,\ell-2} \in \W_{k,\ell-2}$ can be chosen to depend only on $u$
and $\dot u$ with $v_{k+2,\ell} \in \W_{k+2,\ell}$.  Now consider the
replacement
\begin{equation}
  u \mapsto u + \psi^{-2}(X) v_{k+2,\ell}
\end{equation}
with $v_{k+2,\ell} \in \W_{k+2,\ell}$.  Then
\begin{align}
   \tfrac12 \, u \, u_{xx}
   & \mapsto \tfrac12 \, u \, u_{xx}
       + u \, \psi^{-2}(X) \partial_{xx} v_{k+2,\ell}
       + \tfrac12 \, \psi^{-2}(X) v_{k+2,\ell} \,
         \psi^{-2}(X) \partial_{xx} v_{k+2,\ell}
       \notag \\
   & = \tfrac12 \, u \, u_{xx}
       + w_{k+2,\ell-2} + w_{0,2\ell-2} 
\end{align}
for some $w_{k+2,\ell-2} \in \W_{k+2, \ell-2}$ and
$w_{0,2\ell-2} \in \W_{0, 2\ell-2}$.  Next,
\begin{align}
  \ddot u \, v_{k+2,\ell} - \tfrac12 \, u \, \psi^2(X) \ddot u
  & \mapsto - \tfrac12 \, u \, \psi^2(X) \ddot u
      + w_{0,\ell}
\end{align}
with $w_{0,\ell} \in \W_{0,\ell}$.  Finally, when transforming $V$ and
higher-order Lagrangian densities that contain no more than first
derivatives of $u$, then by \ref{v} and \ref{vi}, the additional
remainder term is also of class $\W_{0,\ell}$.  Thus, altogether, one
elimination step results in the replacement
\begin{equation}
  L_{k, \ell} \mapsto L_{k+2,\ell}
  \equiv \tfrac12 \, u \, u_{xx}
        - \tfrac12 \, u \, \psi^2(X) \ddot u
        + V_{k+2,\ell} + R_{k+2,\ell-2}
        + R_{0,\ell}^{k+2}
\end{equation}
with $V_{k+2,\ell} = V_{k,\ell} + F_{k, \ell-2}$.

We iterate this (inner) replacement until $k+2=2m-2$, at which point
we only two time derivatives are left so that we can always choose
$V_{2m,\ell} \sim V_{k+2,\ell} + R_{k+2,\ell-2}$ such that
$V_{2m,\ell}$ depends only on $u$ and $\dot u$, so that
\begin{align}
  L_{2m, \ell}
  & \sim \tfrac12 \, u \, u_{xx}
        - \tfrac12 \, u \, \psi^2(X) \ddot u
        + V_{2m,\ell} + R_{0,\ell}^{2m-2}
        \notag \\
  & \sim \tfrac12 \, u \, u_{xx}
        - \tfrac12 \, u \, \psi^2(X) \ddot u
        + V_{2m,\ell} + F_{0,\ell} + \ddot u \, v_{\ell+2,\ell+2}
\end{align}
where $F_{0,\ell} \in \W_{0,\ell}$ depends only on $u$ and $\dot u$
and $v_{\ell+2,\ell+2} \in \W_{0,\ell+2} = \W_{\ell+2,\ell+2}$.  Now
consider the (outer) replacement
\begin{equation}
  u \mapsto u + \psi^{-2}(X) v_{k+2,\ell+2}
\end{equation}
with $v_{k+2,\ell+2} \in \W_{k+2,\ell+2}$.  Then
\begin{align}
   \tfrac12 \, u \, u_{xx}
   & \mapsto \tfrac12 \, u \, u_{xx}
       + u \, \psi^{-2}(X) \partial_{xx} v_{k+2,\ell+2}
       + \tfrac12 \, \psi^{-2}(X) v_{k+2,\ell+2} \,
         \psi^{-2}(X) \partial_{xx} v_{k+2,\ell+2}
       \notag \\
   & = \tfrac12 \, u \, u_{xx}
       + w_{k+2,\ell} + w_{0,2\ell+2} 
\end{align}
for some $w_{k+2,\ell} \in \W_{k+2, \ell}$ and
$w_{0,2\ell+2} \in \W_{0, 2\ell+2}$.  Next,
\begin{align}
  \ddot u \, v_{k+2,\ell+2} - \tfrac12 \, u \, \psi^2(X) \ddot u
  & \mapsto - \tfrac12 \, u \, \psi^2(X) \ddot u
      + w_{0,\ell+2}
\end{align}
with $w_{0,\ell+2} \in \W_{0,\ell+2}$.  Finally, when transforming $V$
and higher-order Lagrangian densities that contain no more than first
derivatives of $u$, then by \ref{v} and \ref{vi}, the additional
remainder term is also of class $\W_{0,\ell+2}$.  Thus, altogether,
one elimination step results in the replacement
\begin{equation}
  L_{2m,\ell} \mapsto L_{\ell+2,\ell+2}
  \equiv \tfrac12 \, u \, u_{xx}
        - \tfrac12 \, u \, \psi^2(X) \ddot u
        + V_{\ell+2,\ell+2} + R_{\ell+2,\ell} + R_{0,\ell+2}^{\ell+2}
\end{equation}
where $R_{\ell+2,\ell} \in \W(\ell+2,\ell)$ and
$R_{0,\ell+2}^{\ell+2} \in \W_{0,\ell_2}$.  We can then eliminate
$R_{\ell+2,\ell}$ by following the inner replacement loop to its end.

Altogether, we iterate the outer replacement loop until $\ell+2 = 2m-2$,
at which point we can remove all second derivatives and obtain a
final $O(h^{2m})$-modified Lagrangian density of the form
\begin{equation}
  L_{2m,2m} \sim \tfrac12 \, u \, u_{xx}
    - \tfrac12 \, u \, \psi^2(X) \ddot u + V_{2m,2m} 
\end{equation}
where $V_{2m,2m}$ depends only on $u$ and $\dot u$.

\section{Conclusions}
\label{s.conclusions}

Modified equations for backward error analysis of variational
integrators can be systematically constructed using a formal Taylor
expansion of the action integral.  However, a straightforward
variational construction necessitates the use of an extended phase
space.  We have demonstrated for the model case of the implicit
midpoint rule applied to the semilinear wave equation that a carefully
chosen configuration space transformation allows us to eliminate the
dependence of the modified Lagrangian on higher order time
derivatives, thus reducing the phase space, and refitting the modified
equations into the standard framework of Lagrangian mechanics.
Furthermore, such construction yields modified equations whose
dynamics lives on timescales that coincide with the timescales of the
unmodified partial differential equation.  This is clearly not the
case for the modified equations derived on the Hamiltonian side using
the traditional method as can already be seen when discretizing the
linear wave equation.

Even though we looked only at the next-order correction for the
implicit midpoint scheme in detail, we have shown that our
construction generalizes to any order.  We also expect that the
computations shown here generalize in a straightforward way to
semilinear Hamiltonian evolution equations with a linear part
containing a general self-adjoint time-independent operator in place
of $\partial_{xx}$, provided there is a regular Legendre transform.

Our approach was initially developed in the context of Hamiltonian
systems with strong gyroscopic forces \cite{GottwaldOT:2007:LongTA,
Oliver:2006:VariationalAR}; the present work demonstrates that the
strategy is more widely applicable and might point toward an abstract
theory of degenerate variational asymptotics.  We also note that the
flexibility of the approach comes from the fact that the variational
construction modifies the symplectic structure and the Hamiltonian
simultaneously.  In contrast, the traditional construction keeps a
canonical symplectic structure and only modifies the Hamiltonian.

Two major questions remain open.  First, is it possible to show, using
the new variational modified equation, that the implicit midpoint rule
preserves the energy to fourth order?  Note that for the new modified
equations, the truncation remainder is $O(1)$ for $h$ fixed and
$\waven \to \infty$, while it is unbounded in the traditional setting.
Therefore, we expect that we still need some assumptions on the
regularity of the solution---even though numerical simulations
indicate that energy preservation of the implicit midpoint rule is
very good even for non-smooth data---but the regularity conditions are
possibly less stringent than those needed in
\cite{WulffO:2016:ExponentiallyAH}.  Second, is it possible to achieve
exponential asymptotics as in the standard backward error analysis for
ODEs?  To answer this question, it is necessary to describe the
combinatorics of the new asymptotic series, which should be possible,
but is more complicated than the conventional construction.

\section*{Acknowledgments}

We thank Melvin Leok, Haidar Mohamad, Mats Vermeeren and, in
particular, Claudia Wulff for stimulating discussion on the subject of
this paper.  This work was supported in part by DFG grant OL-155/3-2.
This paper is a contribution to project L2 of the Collaborative
Research Center TRR 181 ``Energy Transfers in Atmosphere and Ocean''
funded by the Deutsche Forschungsgemeinschaft (DFG, German Research
Foundation) under project number 274762653.

\bibliographystyle{siam}
\bibliography{bea}

\end{document}